\documentclass[11pt,reqno]{amsart}
\usepackage[left=2.8cm,right=2.8cm,top=2.8cm,bottom=2.8cm]{geometry}
\usepackage[utf8]{inputenc}
\usepackage[T1]{fontenc}
\usepackage[english]{babel}
\usepackage{eurosym}
\usepackage{amssymb}
\usepackage{amstext}
\usepackage{amsmath}
\usepackage{amsfonts}
\usepackage{amsthm}
\usepackage{bm}
\usepackage{color}
\usepackage{enumerate}
\usepackage{mathrsfs}
\usepackage{braket}
\usepackage{csquotes}
\usepackage{amscd,epsfig,esint,graphics}
\usepackage{wrapfig}
\usepackage{verbatim}
\usepackage{mathtools} 
\usepackage{bbm} 
\usepackage{hyperref}
\usepackage[backend=biber, giveninits = true]{biblatex}
\usepackage{svg}

\usepackage{siunitx}
\usepackage{tikz} 
\usepackage{graphicx} 
\usepackage{float} 
\usepackage{todonotes} 

\theoremstyle{plain}
\newtheorem{theorem}{Theorem}[section]
\newtheorem{lemma}[theorem]{Lemma}
\newtheorem{proposition}[theorem]{Proposition}

\newtheorem{remark}[theorem]{Remark}
\newtheorem{definition}[theorem]{Definition}

\theoremstyle{definition}

\theoremstyle{remark}

\numberwithin{equation}{section}

\newcommand{\R}{\mathbb{R}}
\newcommand{\C}{\mathbb{C}}
\newcommand{\K}{\mathbb{K}}

\newcommand{\N}{\mathbb{N}}
\newcommand{\s}{\mathbb{S}}

\newcommand{\M}{\mathcal{M}}
\renewcommand{\H}{\mathcal{H}}
\newcommand{\A}{\mathcal{A}}

\newcommand{\D}{\mathcal{D}}
\renewcommand{\AA}{\mathbb{A}}

\newcommand{\e}{\varepsilon}

\newcommand{\rank}{\textnormal{rank}}

\newcommand{\loc}{\textnormal{loc}}
\newcommand{\Lin}{\textnormal{Lin}}

\newcommand{\supp}{\textnormal{supp}}
\newcommand{\Span}{\textnormal{span}}
\newcommand{\st}{\ : \ }
\newcommand{\dist}{\textnormal{dist}}
\newcommand{\Ima}{\textnormal{Im }}

\newcommand\restr[2]{{
\left.\kern-\nulldelimiterspace 
#1 
\vphantom{\big|} 
\right|_{#2} 
}}

\newenvironment{samurev}{\color{magenta}}{\color{black}}
\newcommand{\bsamr}{\begin{samurev}}
\newcommand{\esamr}{\end{samurev}}

\newenvironment{todorev}{\color{cyan}}{\color{black}}
\newcommand{\btodo}{\begin{todorev}}
\newcommand{\etodo}{\end{todorev}}

\title{A regularity result for $BV^{\A}(\Omega)$}

\author[J. Deutsch]{Jakob Deutsch}
\address[J. Deutsch]{TU Wien, Wiedner Hauptstraße 8-10/104, 1040, Vienna, Austria.}
\email{jakob.deutsch@tuwien.ac.at}

\author[S. Ricc\`o]{Samuele Ricc\`o}
\address[S. Ricc\`o]{TU Wien, Wiedner Hauptstraße 8-10/104, 1040, Vienna, Austria.}
\email{samuele.ricco@tuwien.ac.at}

\bibliography{references.bib}
\begin{document}

\begin{abstract}
    It is well known that distributions whose symmetrized gradient is a bounded Radon measure belong to the space $BD$ on bounded domains with $C^1$ boundary. In this work, we extend this result to a broader class of first-order linear elliptic operators. More precisely, let $\A$ be a first-order linear elliptic operator satisfying the rank-one property. We prove that if a distribution defined on a Lipschitz domain has bounded $\A$-variation, then it belongs to the space $BV^\A$.   \\[0.5cm]
    {\it 2020 Mathematics Subject Classification:} 26B30, 35B65, 35F35, 49N99
    \\
    {\it Keywords and phrases:} fine properties, slicing, bounded $\A$-variation, elliptic regularity
\end{abstract}

\maketitle


\thispagestyle{empty}


\section{Introduction}
\label{sec:intro}

The study of regularity properties for functions and distributions subject to linear constant-rank differential constraints is a central topic in modern analysis, with important connections to the calculus of variations, geometric measure theory, and partial differential equations. In particular, spaces of functions with bounded variation adapted to elliptic differential operators provide a natural framework to investigate weak regularity phenomena beyond the classical Sobolev setting. 

The prototypical example is the space of functions of bounded deformation, $BD(\Omega)$, which corresponds to the case where the differential operator is the symmetrized gradient (cf.\ \cite{temam}). This concept has been extended in various directions by replacing the gradient with more general linear complex elliptic differential operators. More precisely, given a linear differential operator of the form
$$
    \A = \sum_{|\alpha| = k}A_\alpha \partial^\alpha
$$ 
with coefficients $A_\alpha \in {\rm Lin}(V;W)$ for some finite dimensional vector spaces $V$ and $W$, the space $BV^{\mathcal{A}}(\Omega)$ consists of all vector fields in $L^1$ whose $\mathcal{A}$-distributional derivative is a bounded Radon measure. These spaces have proven to be particularly useful in the study of variational problems with linear growth and in the analysis of fine properties of vector fields under differential constraints (see, e.g., \cite{DT84, GR19_0, ARS25}).

In this paper, we focus on first-order elliptic operators $\mathcal{A}$ acting between finite-dimensional vector spaces and satisfying the so-called \emph{rank-one property} introduced by \textsc{A. Arroyo-Rabasa} in \cite{AR20} (cf.\ Definition \ref{def:rank_one}). This structural condition plays an important role in the validity of one-dimensional slicing techniques which have been introduced to study the space of functions of bounded variation (cf.\ \cite{AFP00}). It is closely related to the algebraic mixing condition and can be seen as a natural generalization of properties enjoyed by classical operators like the gradient or the symmetrized gradient (cf.\ \cite{AR20, SVS19}).

Classical regularity results, such as those for functions with bounded deformation (cf.\ \cite[Chapter II, Theorem 2.3]{temam}), typically require the boundary of the domain to be $C^1$. In contrast, the goal of this work is to establish analogous regularity results under weaker geometric assumptions, namely for domains with Lipschitz boundary. In particular, we build on earlier results in the literature concerning trace operators and regularity in $BV^{\A}$ spaces (see \cite{BDG20}) and slicing techniques for the symmetrized gradient \cite{DM13, AR20}. Our main result reads as follows.

\begin{theorem}
\label{thm:DR}
    Let $\Omega \subset \R^n$ be a bounded domain with Lipschitz boundary and let $\A$ be a first-order elliptic operator such that it satisfies the rank-one property as in Definition \ref{def:rank_one}. If $u \in \D'(\Omega; V)$ and $\A u \in \M(\Omega; W)$ then $u \in BV^\A(\Omega; V)$.
\end{theorem}

A key ingredient in the proof is the interplay between the rank-one structure of the operator and the polarization provided by the slicing theory developed in \cite{AR20}. This allows us to reduce the analysis to one-dimensional sections, where classical $BV/BD$ techniques can be applied. 

The paper is organized as follows. In Section \ref{sec:prel}, we introduce the necessary notation and recall basic definitions concerning elliptic operators and $BV^{\mathcal{A}}$ spaces. Section \ref{sec:rank_one} is devoted to a recap of the rank-one property and its equivalent formulations. In Section \ref{sec:compactness}, we establish a compactness result that plays a fundamental role in the proof of the main theorem. We remark here that a similar result has been obtained by \cite{GR26}, but stronger tools for the proof were required. Finally, in Section \ref{sec:regularity}, we combine these tools to prove the desired regularity result.


\subsection{Notation}
\label{sec:notation}

\setlength{\parindent}{0pt} 

Given $V$, a finite-dimensional euclidean space of dimension $N$, and $\Omega \subset \R^n$ an open set, we denote with $\M(\Omega;V)$ the set of bounded $V$-valued Radon measures on $\Omega$ and with $\D'(\Omega; V)$ the set of $V$-valued distributions. We denote with $| \cdot | : \M(\Omega;V) \to [0,+\infty)$ the total variation of a measure. The notation $\langle \cdot , \cdot \rangle$ denotes both the duality pairing of $V$ with its dual space $V^*$, and the scalar product in $\mathbb{R}^n$. We indicate the norm of a vector in the space $V$, or its dual space $V^*$, by $| \, \cdot  \, |$. If $\tilde V \subset V$ is a subspace of $V$ we will write $\tilde V \leq V$, and we denote by $P_{\tilde V}$ the orthogonal projection on $\tilde V$. 


\section{Preliminaries}
\label{sec:prel}

Let us consider a homogeneous linear partial differential operator $\A$ of order $k \in \N$ on $\R^n$ with constant coefficients, namely
\begin{equation*}
    \A : \D'(\R^n; V) \to \D'(\R^n; W),
\end{equation*}
where $V$ and $W$ are finite-dimensional Euclidean spaces of dimensions $N$ and $M$, respectively. More rigorously, for a multi-index $\alpha \in (\N_0)^n$ with order $|\alpha| = \alpha_1 + \ldots + \alpha_n = k$, we consider differential operators of the form
\begin{equation}
\label{def:operator}
    \A u := \sum_{|\alpha| = k} A_\alpha \partial^\alpha u \qquad \textnormal{ for } u \in \D'(\R^n; V),
\end{equation}
where $\partial^\alpha$ represents the distributional derivative $\partial^{\alpha_1} \ldots \, \partial^{\alpha_n}$ and where the coefficients $A_\alpha \in W \otimes V^* \cong \Lin(V;W)$ are constant tensors. Given $\K \in \{\R, \C\}$, we call the {\it principal symbol} associated to the operator $\A$ the tensor-valued $k$-homogeneous polynomial $\AA : \K^n \to \Lin(V;W)$ defined as
\begin{equation*}
    \AA(\xi) := \sum_{|\alpha| = k} \xi^\alpha A_\alpha \qquad \textnormal{ for } \xi \in \K^n,
\end{equation*}
where $\xi^\alpha = \xi_1^{\alpha_1} \ldots \, \xi_n^{\alpha_n}$.

\begin{definition}
    Let $\K \in \{\R, \C\}$. Given an operator $\A$ as in \eqref{def:operator}, we say that $\A$ has constant rank if and only if there exists $r \in \N$ such that
    \begin{equation*}
        \rank \, \AA(\xi) = r \qquad \textnormal{ for every } \xi \in \K^n \setminus \{0\}.
    \end{equation*}
    Equivalently, $\A$ has constant rank if and only if $\dim(\Ima(\AA(\xi))) = r$ independently of $\xi \in \K^n \setminus\{0\}$.
\end{definition}

\begin{definition}
    Let $\K \in \{\R, \C\}$. Given an operator $\A$ as in \eqref{def:operator}, we say that $\A$ is $\K$-elliptic if and only if
    \begin{equation*}
        \ker(\AA(\xi)) = \{0\} \qquad \textnormal{ for every } \xi \in \K^n \setminus \{0\}.
    \end{equation*}
    Equivalently, $\A$ is $\K$-elliptic if and only if there exists a positive constant $C$ such that
    \begin{equation*}
        |\AA(\xi)v| \ge C |\xi|^k|v| \qquad \textnormal{ for every } \xi \in \K^n \textnormal{ and } v \in V.
    \end{equation*}
\end{definition}

The concept of complex-ellipticity has been introduced in \cite{S61,S70} and has been proven to be a necessary and sufficient condition for the existence of trace operators on $BV^\A$-spaces, see \eqref{BVA}, when $\A$ is a first-order operator. The next definition, first introduced in \cite{SVS19}, is about an algebraic mixing condition on the principal symbol of our operator.

\begin{definition}
    Let $\A$ be an $\R$-elliptic operator as in \eqref{def:operator}. We say that the principal symbol of $\A$ satisfies the algebraic mixing property \eqref{mixing} if and only if
    \begin{equation}
    \label{mixing}\tag{m}
        \bigcap_{\substack{\pi \, \leq\,  \R^n \\  \dim{\pi}\, =\, n-1}} \Span\{ \AA(\eta)v \st \eta \in \pi, \, v \in V \} = \{0\}.
    \end{equation}
\end{definition}

Let us now state the following lemma relating the ellipticity concept with the algebraic mixing property. For the second implication we refer to \cite[Remark 3.1]{AR20}.

\begin{lemma}
\label{lmm:ellipticities}
    Let $\A$ be an operator as in \eqref{def:operator}. Then, we have that
    \begin{equation*}
        \A \textnormal{ is } \C \textnormal{-elliptic } \quad \Rightarrow \quad \A \textnormal{ is } \R \textnormal{-elliptic}.
    \end{equation*}
    If, in particular, $\A$ is a first-order operator, then
    \begin{equation*}
        \A \textnormal{ is } \R \textnormal{-elliptic and satisfies \eqref{mixing}} \quad \Rightarrow \quad \A \textnormal{ is } \C \textnormal{-elliptic}.
    \end{equation*}
\end{lemma}

\begin{remark}
    Note that the converses of both implications in Lemma \ref{lmm:ellipticities} are not true, see respectively \cite[Example 2.2 (c)]{BDG20} and \cite[Examples 10.1 and 10.2]{AR20}.
\end{remark}

We now need to recall the definition of the function space $BV^\A$, introduced in \cite{GR19_0}. Let $\Omega \subset \R^n$ be open and let $\A$ be defined as in \eqref{def:operator}. 
Then, the space $BV^\A$ of functions of bounded $\A$-variation is defined as
\begin{equation}
\label{BVA}
    BV^\A(\Omega;V) := \{u \in L^1(\Omega;V) \st \A u \in \M(\Omega;W)\}.
\end{equation}
Note that this is a Banach space if endowed with the norm
\begin{equation*}
    \|u\|_{BV^\A(\Omega;V)} := \|u\|_{L^1(\Omega;V)} + |\A u|(\Omega).
\end{equation*}
If we consider $\A$ to be a $\C$-elliptic differential operator of order $k$, then the space $BV^\A$ is a generalization of the $\nabla^k B$ space introduced in \cite{DT84}. We recap below the notion of directional spectrum, which is essential when discussing the slicing technique we use. 
\begin{definition}\label{def:directional_spectrum}
    The {directional spectrum} $\partial \sigma(\A)$ is defined as the set of every $(\xi, v^*) \in \R^n \times V^*$ such that there exists $w^* \in W^*$ for which 
    \[
        \langle w^*, \A(\eta)v\rangle = \langle \xi, \eta \rangle \langle v^*, v \rangle
    \]
    holds for every $(\eta, v) \in \R^n \times V$.
\end{definition}

Now, we present the slicing theorem technique that we use for proving our result (cf.\ \cite[Proposition 2.1]{AR20}). We use the notation
\[
    \Omega_y^\xi := \{ t \in \R \ : \ y + t\xi \in \Omega \}
\]
for $\xi \in \s^{n-1}$ and $y \in \{\xi \}^\perp$.

\begin{proposition}
\label{prop:slicing_theory}
    Let $\A : C^\infty(\R^n; V) \to C^\infty(\R^n; W)$ be a first-order, homogeneous and elliptic differential operator satisfying the algebraic mixing property $(\textnormal{m}_1)$. Let $u \in BV^\A(\Omega;V)$ and suppose that there exists $(\xi, v^*) \in \partial \sigma(\A)$ with $w^* \in W^*$ such that 
     \[
        \langle w^*, \A(\eta)v\rangle = \langle \xi, \eta \rangle \langle v^*, v \rangle
    \]
    holds for every $(\eta, v) \in \R^n \times V$.
    Then, we have 
    \[
        u^{v^*}_{y,\xi} := (t \mapsto \langle v^*, u(y + t \xi)\rangle) \in BV(\Omega_y^\xi) 
    \]
    for $\H^{n-1}$-almost every $y \in \{\xi\}^\perp$ and 
    \[
        \int_{\{\xi\}^\perp} |D u^{v^*}_{y,\xi}|(\Omega_y^\xi) \, dy \leq |\A u|(\Omega)|w^*|.
    \]
\end{proposition}


\section{Rank-one condition}
\label{sec:rank_one}

In this section, we give a short recap and introduction of the rank-one condition introduced by \textsc{Arroyo-Rabasa} in \cite{AR20}. Let $\A$ be a first-order linear partial differential operator as in \eqref{def:operator}, namely $\A : \D'(\R^n; V) \to \D'(\R^n; W)$ defined by
\begin{equation}
\label{def:oper_first_ord}
    \A := \sum_{i = 1}^n A_i \partial_i
\end{equation}
where $A_i : V \to W$ for every $i = 1, \ldots, n$ are linear maps between the two euclidean spaces $V$ and $W$. We define the bilinear map
\[  
    f_\A(\xi, v) := \AA(\xi) v \quad \textnormal{ for } \xi \in \R^n \textnormal{ and } v \in V,
\]
and denote by $\overline{f_\A}$ its extension to $\R^n \otimes V$. Furthermore, let
\begin{equation}
\label{def:pullback_directional_spectrum}
    g_\A : 
    \begin{dcases}
        W^* \to (\R^n \otimes V)^* \cong \R^n \otimes V^*, \\
        w^* \mapsto w^* \circ \overline{f_\A}.
    \end{dcases}
\end{equation}
We call $w^* \in W^*$ a {\it rank-one vector} if $g_\A(w^*)$ corresponds to a rank-one tensor or, more precisely, if there exists $\xi \in \R^n$ and $v^* \in V^*$ such that the relation
\[
    \langle g_\A(w^*), \eta \otimes v \rangle = \langle \xi, \eta \rangle \langle v^*,v \rangle
\]
holds for every $0 \ne \eta \in \R^n$ and $0 \ne v \in V$. We denote the set of rank-one vectors by $\A_1^\otimes$. Moreover, we define 
\[
    W_\A := \Span \left( \bigcup_{\xi \in \s^{n-1}} \Ima \AA(\xi) \right).
\]

\begin{definition}
\label{def:rank_one}
    Let $\A$ be defined as in \eqref{def:oper_first_ord}. We say that $\A$ fulfills the rank-one property, or equivalently that $\A$ is rank-one, if
    \[
        \Span \, \A_1^\otimes = (W_\A)^*.
    \]
\end{definition}

If $\A$ is rank-one we can freely pass between the notions of rank-one vectors and directional spectrum since 
\[
    \partial \sigma(\A) = g_\A(\A_1^\otimes), 
\]    
where $\partial \sigma(\A)$, namely the directional spectrum of $\A$, is defined in Definition \ref{def:directional_spectrum}.  \\[0.5em]
The following two statements are direct consequences of the theory recalled so far.
\begin{lemma}
\label{lem:injectivity_of_pullback_and_2d_property}
    Let $\A$ be defined as in \eqref{def:oper_first_ord}. We have the following:
    \begin{enumerate}
        \item[(i)] The map $g_\A$ defined in \eqref{def:pullback_directional_spectrum} is injective on $(W_\A)^*$.
        
        \item[(ii)] Suppose that $\A$ fulfills the rank-one property and consider the two-dimensional subspaces $\Pi \subset \R^n$, $\widetilde V \subset V$. Then, the restricted differential operator $\widetilde \A$, induced by the symbol $\AA(\xi)v$ with $\xi \in \Pi$ and $v \in \widetilde V$, has the rank-one property.
    \end{enumerate}
\end{lemma}
\begin{proof}
    {\it (i)}: Indeed, let $w^*, \bar w^* \in (W_\A)^*$ with $w^* \ne \bar w^*$. This means that there exist $\xi \in \R^n$ and $v \in V$ such that 
    \[
        \langle w^* \circ f_\A, \xi \otimes v \rangle = \langle w^*, \AA(\xi)v \rangle \ne \langle \bar w^*, \AA(\xi)v \rangle = \langle \bar w^* \circ f_\A, \xi \otimes v \rangle.
    \]\ \\[-1em]
    {\it (ii)}: Let $w^* \in (W_{\widetilde \A})^* \subset (W_\A)^*$. By the rank-one property, $w^*$ can be written as a linear combination  of $w_i^* \in \A_1^\otimes$ with corresponding $\eta_i \in \R^n$ and $v_i^* \in V^*$. But for any $\xi \in \Pi$ and $v \in \widetilde V$ 
    \[
        \langle \mathcal{P}_{(W_{\widetilde \A})^*} w_i^*, \AA(\xi)v \rangle = \langle w_i^*, \AA(\xi)v \rangle = \langle \eta_i, \xi\rangle  \langle v_i^*, v \rangle = \langle \mathcal{P}_\Pi \eta_i, \xi \rangle \langle \mathcal{P}_{\widetilde V^*} v_i^*, v \rangle,
    \]
    where, given $E$ and $F$ finite-dimensional vector spaces, $\mathcal{P}_{E^*} f^* := f^* \circ \iota \in E^*$ for $f^* \in F^*$ with $\iota$ being the inclusion of $E$ in $F$. 
    In particular, $w^* = \mathcal{P}_{(W_{\widetilde \A})^*} w^*$ is a linear combination of $\mathcal{P}_{(W_{\widetilde \A})^*} w_i^*$ which are rank-one, which in turn implies $(W_{\widetilde \A})^* \subset \Span \, \widetilde \A_1^\otimes$.
\end{proof}

It has been shown in \cite[Lemma 4.1]{AR20} that we have the following characterization of the rank-one condition. We recall the proof for convenience.

\begin{lemma}
    Let $\A$ be defined as in \eqref{def:oper_first_ord}. Then the following are equivalent: 
    \begin{enumerate}
        \item[(i)] $\A$ satisfies the rank-one property,
        \item[(ii)] $\A$ satisfies the mixing condition $(\ref{mixing})$.
    \end{enumerate}
\end{lemma}
\begin{proof}
    Let us begin by proving ${\it (i)} \implies {\it (ii)}$. Suppose that 
    \begin{equation}
    \label{eq:mixing_intersection}
        w \in \bigcap_{\substack{\pi \, \leq\,  \R^n \\  \dim{\pi}\, =\, n-1}} \Span \{ \AA(\eta)v \st \eta \in \pi, \, v \in V \}.
    \end{equation}
    Let $w^* \in \A_1^\otimes$, and let $\eta \in \s^{n-1}$ and $v^* \in V^*$ be such that
    \[
        g_\A(w^*) = \eta \otimes v^*.
    \]
    By \eqref{eq:mixing_intersection} we then find $\tilde \xi \in \{\eta\}^\perp \cap \s^{n-1}$ and $\tilde v \in V$ such that $w = \AA(\tilde \xi) \tilde v$. Now, we can infer
    \[
        \langle w^*, w \rangle = \langle \tilde \xi, \eta\rangle \langle v^*, v \rangle = 0.
    \]
    Since $\Span\, \A_1^\otimes = (W_\A)^*$ holds by the rank-one condition, it follows that $\langle w^*, w \rangle = 0$ for every $w^* \in (W_A)^*$ which, in particular, implies $w = 0$.
    \\
    In order to prove that ${\it (ii)} \implies {\it (i)}$, let us consider $w^* \in \bigcap_{\eta \in \pi} \ker(\AA(\eta)^*)$ for $\pi \leq \R^n$ with $\dim(\pi) = n - 1$. This implies
    \[
        \restr{(w^* \circ \overline{f_\A})}{\pi \otimes V} = 0.
    \]
    In particular, we have $g_\A(w^*) \in \pi^\perp \otimes V^*$ which implies that $w^*$ is a rank-one vector. Now, the mixing condition directly implies
    \begin{equation*}
    \begin{aligned}
        (W_\A)^*
        &= (\{0\})^\perp \\
        &= \Bigg(\bigcap_{\substack{\pi \, \leq\,  \R^n \\  \dim{\pi}\, =\, n-1}} \Span\{ \AA(\eta)v \st \eta \in \pi, \, v \in V \} \Bigg)^\perp \\
        &= \underset{\substack{\pi \, \leq\,  \R^n \\  \dim{\pi}\, =\, n-1}}{\Span} \left( \bigcup_{\eta \in \pi} \Ima \AA(\eta) \right)^\perp \\
        &= \underset{\substack{\pi \, \leq\,  \R^n \\  \dim{\pi}\, =\, n-1}}{\Span} \left( \bigcap_{\eta \in \pi} \ker \left( \AA(\eta)^* \right) \right),
    \end{aligned}
    \end{equation*}
    with which we conclude.
\end{proof}

Now, we state the polarization \cite[Proposition 4.4]{AR20} property of first-order operators that are elliptic and satisfy the rank-one property, which is essential in proving the compactness and regularity theorems in Section \ref{sec:compactness} and Section \ref{sec:regularity}. We provide a short proof for the convenience of the reader.

\begin{lemma}\label{lem:polarization}
    Let $\A$, defined as in \eqref{def:oper_first_ord}, be elliptic and satisfy the rank-one property. Suppose that $(\xi, e), (\eta, f) \in \partial \sigma(\A)$. Then, there exists $\gamma \in \R$ with $\gamma \neq 0$ such that 
    \[
        (\xi + \lambda \eta, e + \gamma \lambda f) \in \partial \sigma (\A)
    \]
    for every $\lambda \in \R$.
\end{lemma}

\begin{proof}
    Let us assume that $\xi$ and $\eta$ are linearly independent in $\R^n$ and $e$ and $f$ are linearly independent in $V^*$; otherwise, the statement is trivial as you can choose any $\gamma \neq 0$. By Lemma \ref{lem:injectivity_of_pullback_and_2d_property}{\it (ii)}, it is enough to consider the restriction of $\A$ to $\Span\{ \xi, \eta\}$ and $\Span\{ e,f \}$ as the restricted differential operator still fulfills the rank-one condition. In particular, we can assume $n = 2$ and $\dim V = 2$. Notice that $\dim W_\A > 2$ because of ellipticity and the rank-one condition (mixing condition). Since $g_\A$ is injective on $(W_\A)^*$, cf. Lemma \ref{lem:injectivity_of_pullback_and_2d_property}{\it (i)}, we also have that $\dim (W_\A)^* \leq 4$. If $\dim (W_\A)^* = 4$ we have 
    \[
        g_\A((W_\A)^*) = \R^n \otimes V^*
    \]
    which, in particular, implies $\partial \sigma(\A) = \R^n \times V^*$ from which we derive the existence of a continuous function $h \subset \partial \sigma(\A)$. \\
    If $\dim W_\A = 3$ we have that $g_\A((W_A)^*)$ can be spanned by the three vectors
    \[
        \xi \otimes e, \, \eta \otimes f,\, \alpha (\xi \otimes f) + \beta (\eta \otimes e)
    \]
    for some $\alpha, \beta \in \R$. Furthermore, we notice that for $a, b, c, d \in \R^n$ we have
    \[
        (a \xi + b \eta) \otimes (c e + d f) = ac (\xi \otimes e) + bf (\eta \otimes f) + ad (\xi \otimes f) + bc (\eta \otimes e).
    \]
    If $a, b$ are given, we can choose $d = \alpha b$ and $c = \beta a$. With these choices of coefficients, we have
    \[
        (a \xi + b \eta) \otimes (c e + d f) \in g_\A((W_A)^*).
    \]
    Setting $a = 1$, $b = \lambda$, $\gamma = \alpha/\beta$ and dividing by $\beta$ we derive 
    \[
        (\xi + \lambda \eta) \otimes (e + \gamma \lambda f) \in g_\A((W_A)^*).
    \]
\end{proof}

\begin{lemma}\label{lem:directional_spectrum_surjectivity_projections}
    Let $\A$, defined as in \eqref{def:oper_first_ord}, be elliptic and satisfy the rank-one property. We have $P_{\R^n}(\partial \sigma(\A)) = \R^n$ and $P_{V^*}(\partial \sigma(\A)) = V^*$.
\end{lemma}
\begin{proof}
    Reducing again to $n = 2$ and $\dim V = 2$ we infer the lemma from Lemma \ref{lem:polarization}. As before, by restricting the differential operator to planes, we derive the result for arbitrary $n$ and $\dim V$.
\end{proof}


\section{Compactness}
\label{sec:compactness}

In this section, we prove the compactness result. The proof of this result is based on a combination of ideas from \cite{AR20} and \cite[Section 10]{DM13}.

\begin{lemma}
\label{lem:translation_lemma}
    Let $\A$, defined as in \eqref{def:oper_first_ord}, be elliptic and satisfy the rank-one property. Let $K > 0$ and let $u \in L^1(\R^n; V)$ be such that 
    $$\| u \|_{L^1(\R^n; V)} < K.$$ 
    Suppose that for every $(\xi, v^*) \in \partial \sigma(\A) \cap (\s^{n-1} \times V^*)$ with $|v^*| = 1$ there exists $\omega_{(\xi, v^*)}: \R^+ \to \R^+$ non-decreasing with $\lim_{h \to 0} \omega_{(\xi, v^*)}(h) = 0$ such that for every $h \in (0,1)$ we have
    \[
        \int_{\R^n} |\langle v^*, u(x + h \xi)\rangle - \langle v^*, u(x) \rangle| \, dx \leq \omega_{(\xi, v^*)}(h).
    \]
    Then, there exist $h_0 \in (0,1)$ and $\tilde \omega : \R^+ \to \R^+$ non-decreasing with $\lim_{h \to 0} \tilde \omega(h) = 0$ (only depending on $K$ and finitely many $\omega_{(\xi, v^*)}$) such that for every $h \in (0, h_0)$ we have
    \begin{align}\label{eq:to_show}
        \int_{\R^n} |u(x + h \xi)- u(x)| \, dx \leq \tilde \omega(h).
    \end{align}
\end{lemma}

\begin{proof}
    Let $\xi \in \s^{n-1}$. Assume first that  
    \[
        \int_{\R^n} |u(x \pm h e_i)- u(x)| \, dx \leq \tilde \omega_i(|h|)
    \]
    holds for every $i = 1,\ldots, n$ and every $h \in (0,1)$ for some modulus of continuity $\tilde \omega_i$.
    Then,
    \begin{align}
    \label{eq:to_show_2}
        \int_{\R^n} |u(x + h \xi) - u(x)| \, dx 
        &\leq \sum_{j = 1}^N \int_{\R^n} \left|u \left(x + h\sum_{i = 1}^j \xi_i e_i \right) - u \left(x + h\sum_{i = 1}^{j -1} \xi_i e_i \right) \right| \, dx \nonumber \\
        &= \sum_{j = 1}^N \int_{\R^n} |u(x + h\xi_j e_j) - u(x)| \, dx  \nonumber \\
        &\leq \sum_{j = 1}^N \tilde \omega_j(h|\xi_j|)  \nonumber \\
        &\leq \sum_{j = 1}^N \tilde \omega_j(h).
    \end{align}
    It therefore suffices to only check \eqref{eq:to_show} for a finite number of $\xi$.
    \\
    Let $\xi \in \s^{n-1}$. Choose $v^* \in V^*$ such that $|v^*| = 1$ and $(\xi, v^*) \in \partial \sigma (\A)$, which is possible due to Lemma \ref{lem:directional_spectrum_surjectivity_projections}. Moreover, choose $(w_i^*)_{i = 1}^{N-1}$ such that $v^*, w_1^*, ..., w_{N - 1}^*$ form an orthonormal basis of $V^*$. Now, we again apply Lemma \ref{lem:directional_spectrum_surjectivity_projections} to get $\xi_{i} \in \s^{n-1}$ such that $(\xi_i, w_i^*) \in \partial \sigma(\A)$. Then,
    \[
        |u(x + h\xi) - u(x)| \leq |\langle v^*, u(x + h\xi) - u(x) \rangle| + \sum_{i = 1}^{N - 1}|\langle w_i^*,u(x + h\xi) - u(x)\rangle |.
    \]
    We directly estimate 
    \begin{align*}
        &|\langle w_i^*, u(x + h\xi) - u(x)\rangle| \\
        \leq ~ & |\langle w_i^*, u(x + h\xi) - u(x + \sqrt h \xi_{i})\rangle| + |\langle w_i^*, u(x + \sqrt h \xi_{i}) - u(x)\rangle|.
    \end{align*}
    Now, define $\zeta_i \in \s^{n-1}$ and $s_h^i \in \R$ for every $i = 1, \ldots, N-1$ via the equation
    \[
        x + h\xi = x + \sqrt h \xi_{i} + s_h^i \zeta_{i}
    \]
    or, equivalently,
    \[
        \zeta_i = (\xi_i + \sqrt h \xi)\frac{\sqrt h}{s_h^i}.
    \]
    By Lemma \ref{lem:polarization} for every $i = 1, \ldots, N-1$ there exists $\gamma_i \in \R$ such that $g_i^* := (w_i^* + \gamma_i \sqrt h v^*)/|w_i^* + \gamma_i \sqrt h v^*|$ fulfills $(\zeta_i, g_i^*) \in \partial \sigma(\A)$. We now notice that $|s_h^i| \leq 2\sqrt h$ from which we further infer $|\zeta_i - \xi_i| \leq 2\sqrt h$. By construction, we also have $|w^*_i - g^*_i| \leq L|\gamma_i||\sqrt h|$ for a constant $L > 0$ and small $h > 0$. From this, we now derive  
    \begin{align*}
        &|\langle w_i^* - g_i^*, u(x + h\xi)\rangle| + |\langle g_i^*, u(x + h\xi) - u(x + \sqrt h \xi_{i})\rangle| + |\langle w_i^* - g_i^*, u(x + \sqrt h \xi_{i})\rangle| \\
        \leq \ & (|u(x + h\xi)| + |u(x + \sqrt h\xi_i)|)|\gamma_i| \sqrt h + |\langle g_i^*, u(x + \sqrt h \xi + s_h^i \zeta_i) - u(x + \sqrt h \xi) \rangle|.
    \end{align*}
    Integrating gives 
    \begin{align*}
        &\int_{\R^n} |u(x + h\xi) - u(x)| \, dx \\
        & \hspace{4em} \leq \sum_{i = 1}^N \int_{\R^n} |\langle g_i^*, u(x + \sqrt h \xi + s_h^i \zeta_i) - u(x + \sqrt h \xi) \rangle | \, dx + 2 K \left(\sum_{i = 1}^N|\gamma_i| \right)\sqrt h \\
        & \hspace{4em} = \sum_{i = 1}^N \int_{\R^n} |\langle g_i^*, u(x + s_h^i \zeta_i) - u(x) \rangle | \, dx + 2 K \left(\sum_{i = 1}^N|\gamma_i| \right)\sqrt h \\
        & \hspace{4em} \leq \sum_{i = 1}^N \omega_{(\zeta_i, g_i^*)}(h) + 2 K \left(\sum_{i = 1}^N|\gamma_i| \right)\sqrt h. \\
    \end{align*}
    This yields \eqref{eq:to_show} with $\omega_\xi(h) := \sum_{i = 1}^N \omega_{(\zeta_i, g_i^*)}(h) + 2 K \left(\sum_{i = 1}^N|\gamma_i| \right)\sqrt h$ for an arbitrary $\xi \in \s^{n-1}$. Defining 
    \[
        \tilde \omega := \sum_{i = 1}^N \tilde \omega_{e_i}(h),
    \]
    we infer \eqref{eq:to_show} for general $\xi \in \s^{n-1}$ from \eqref{eq:to_show_2}.
\end{proof}

As a direct consequence, we now derive a compactness theorem for $BV^\A(\Omega; V)$.

\begin{theorem}
\label{thm:compactness}
    Let $\Omega \subset \R^n$ be a Lipschitz domain, and let $\A$ be defined as in \eqref{def:oper_first_ord}, be elliptic and satisfy the rank-one property. Then, for every sequence $\{u_m\}_{m \in \N} \subset BV^\A(\Omega; V)$ with 
    \[
        \| u_m \|_{BV^\A(\Omega;V)} < K
    \]
    there exist a subsequence $\{u_{m_j}\}_j$ and $u \in BV^\A(\Omega; V)$ with 
    \[
        u_{m_j} \to u
    \]
    in $L^1(\Omega; V)$. 
\end{theorem}
\begin{proof}    
    By the trace theorem for $\C$-elliptic operators (cf.\ \cite{BDG20}), we know that we can extend $u_m$ to $\R^n$ (without renaming it). Now, let $\phi_\e$ be a mollifier, choose $\e_m \to 0^+$ such that for $v_m = u_m \ast \phi_{\e_m} \in C_c(\R^n; V)$ we have
    \begin{align}\label{eq:approximation}
        \| v_m - u_m \|_{L^1(\R^n, V)} + \big||Dv_m|(\R^n) - |Du_m|(\R^n) \big| \leq \frac{1}{m}.
    \end{align}
    In particular, 
    \[
        |Dv_m|(\R^n) \leq \tilde K 
    \]
    for some large enough constant $\tilde K$. For $h \in (0,1)$ and $w \in \{u_m\}$, we apply Fubini's Theorem and Theorem \ref{prop:slicing_theory} to derive 
    \begin{align*}
        &\int_{\R^n} |\langle v^*, w(x + h \xi) - w(x) \rangle| \, dx \\
        &= \int_{\{\xi\}^\perp} \int_\R |\langle v^*, w(y + (t + h)\xi) - w(y + t\xi) \rangle| \, dt \, dy \\
        &\leq \int_{\R^{n-1}} \int_\R \int_t^{t + h} |(w^{v^*}_{y, \xi})'(s)| \, ds \, dt \, dy \\
        &\leq h C \int_{\R^N} |\A w| \, dx \\
        &\leq h C\tilde K
    \end{align*}
    We have shown that 
    \[
        \omega_{(\xi, v^*)} : = \left(h \mapsto h C\tilde K \right)
    \]
    is a suitable modulus of continuity for Lemma \ref{lem:translation_lemma}. Therefore, applying Lemma \ref{lem:translation_lemma}, we infer for every $m \in \N$
    \[
        \int_{\R^n} |v_m(x + h \xi)- v_m(x)| \, dx \leq \tilde \omega(h)
    \]
    for a modulus of continuity $\tilde\omega$. Applying the Fréchet-Kolmogorov compactness theorem, we infer the existence of $u$ such that $v_m \to u$ in $L^1(\R^n; V)$ up to a subsequence (without renaming). Since $\A v_m$ are uniformly bounded as measures they converge weakly$^*$ to some $\mu \in \M_b(\R^n; W)$ and since $v_m \to u$ in $D'(\R^n; V)$ we infer $\A u = \mu$, i.e., $u \in BV^\A(\R^n; V)$. Recalling \eqref{eq:approximation}, we conclude.
\end{proof}

\begin{remark}\label{remark:poincare_embedding}
    In \cite[Theorem 1.1]{GR19_0} it was shown that for $p \leq n/(n-1) = 1^*$
    \begin{align}\label{eq:embedding}
        \|u\|_{L^p(\Omega; V)} \leq C (\|u\|_{L^1(\Omega; V)} + |\AA u|(\Omega))
    \end{align}
    holds for $u \in BV^\A(\Omega; V)$ for $\Omega = B$ for a ball $B$. By using \eqref{eq:embedding}, the compactness result can be transferred to $L^p(\Omega; V)$. Indeed, applying the generalised Hölder inequality 
    \[
        \|u\|_{L^p} \leq \|u\|_{L^1}^\alpha\|u\|_{L^{1^*}}^{(1-\alpha)},
    \]
    for 
    \[
        \alpha = \frac{\frac{1}{p} - \frac{1}{1^*}}{1 - \frac{1}{1^*}}
    \]
    one can immediately deduce that bounded sequences in $BV^\A(\Omega;V)$ that converge in $L^1(\Omega; V)$ also converge in $L^p(\Omega; V)$ for $p \in [1, 1^*)$. As another direct consequence of \eqref{eq:embedding}, we obtain for any two open sets $\tilde \Omega, \Omega$ with $\tilde \Omega \subset \subset \Omega$
    \[
        \| u \|_{L^{1^*}(\tilde \Omega)} \leq C(\tilde \Omega) (\|u\|_{L^1(\Omega; V)} + |\AA u|(\Omega)).
    \]
\end{remark}


\section{Proof of the main result}
\label{sec:regularity}

In this section, we prove the main result of this paper, namely Theorem \ref{thm:DR}. As a first step, we establish a local regularity result for distributions with $\A u$ being a bounded measure.

\begin{theorem}
\label{thm:local_reg}
    Let $\Omega \subset \R^n$ be a bounded domain and let $\A$, defined as in \eqref{def:oper_first_ord}, be elliptic and satisfy the rank-one property. If $u \in \D'(\Omega; V)$ and $\A u \in \M(\Omega; W)$ then $u \in BV_\loc^\A(\Omega; V)$.
\end{theorem}
\begin{proof}
    By Lemma \ref{lmm:ellipticities}, it follows that $\A$ is $\C$-elliptic. By \cite[Proposition 3.1]{GR19_0} we know that assuming that $\A$ is $\C$-elliptic is equivalent to the fact that there exists an integer $l = l(\A)$ such that distributions satisfying $\A u = 0$ in $\Omega$ are polynomials of degree at most $l$. Set
    \begin{equation*}
        F := \{ \textnormal{polynomials of degree at most } l \}
    \end{equation*}
    and consider $\omega \subset \subset \Omega$. Observe that $F \subset L^2(\omega) \cap \mathcal{C}^\infty(\R^n)$, thus we can consider an orthonormal basis $\{b_i\}_{i = 0}^m$ of $F$ with respect to $L^2(\omega)$. Now, let $\{\psi^k\}_k \subset \mathcal{C}^\infty_c(\Omega; [0,1])$ be a family of bump functions such that 
    $$
        \psi^k \searrow \chi_\omega
    $$
    pointwise with $\psi^k = 1$ on $\omega$ and $\psi^k(x) = 0$ for every $x \in \Omega$ such that $\dist(x,\omega) > \tfrac{1}{k}$. Note that this implies
    \begin{equation*}
        \lim_{k \to +\infty} \langle \psi^k b_j, b_i \rangle_{L^2(\Omega)} = \delta_{ij}
    \end{equation*}
    for every $i,j \in [0,m] \cap \N_0$. For $w \in L^1(\Omega; V)$ we define the continuous seminorm 
    \begin{equation*}
        \rho(w) := \sum_{i = 0}^m \sum_{k \in \N} \frac{1}{2^kc_{i,k}} \left| \int_\Omega \psi^k b_iw \, dx \right|.
    \end{equation*}
    where 
    \[
        c_{i,k} := \max\{1, \langle w, \psi^k b_i \rangle\}.
    \]
    Let us remark that for $w \in L^1(\Omega;V) \cap L^2(\Omega; V)$ we have
    \begin{equation*}
    \begin{aligned}
        \rho(w) = 0
        &\quad \iff \quad | \langle \psi^k b_i, w \rangle_{L^2(\Omega)} | = 0 \quad \forall \, i \in [0,m] \cap \N_0, \forall \, k \in \N \\
        &\quad \implies \quad | \langle b_i, w \rangle_{L^2(\omega)} | = 0 \quad \forall \, i \in [0,m] \cap \N_0 \\
        &\quad \iff \quad w|_{\omega} \perp F,
    \end{aligned}
    \end{equation*}
    meaning that $\ker \A \cap \ker \rho = \emptyset$. Note that $\rho$ can be defined for $v \in L^1(\tilde \Omega, V)$ for any $\tilde \Omega \subset \Omega$ with 
    \[
        \rho(v) := \rho(\tilde v)
    \]
    where $\tilde v$ is
    \[
        \tilde v(x) := \begin{cases}
            v(x) & x \in \tilde \Omega, \\
            0 & x \in \Omega \setminus \tilde \Omega.
        \end{cases}
    \]
    In particular, by an application of Theorem \ref{thm:compactness}, we have for Lipschitz domains $\tilde \omega, \tilde \Omega$ with $\omega \subset \subset \tilde \omega \subset \subset \tilde \Omega \subset \Omega$ that for every $v \in BV^\A(\widetilde \Omega; V)$
    \begin{align}\label{eq:poincare_bva_2}
        \| v \|_{BV^\A(\tilde \Omega; V)} \leq C(\tilde \Omega)(\rho(v) + |\A v|(\widetilde \Omega))
    \end{align}
    and, consequently, by Remark \ref{remark:poincare_embedding} we have
    \begin{figure}
        \centering
        \includegraphics[width=0.6\linewidth]{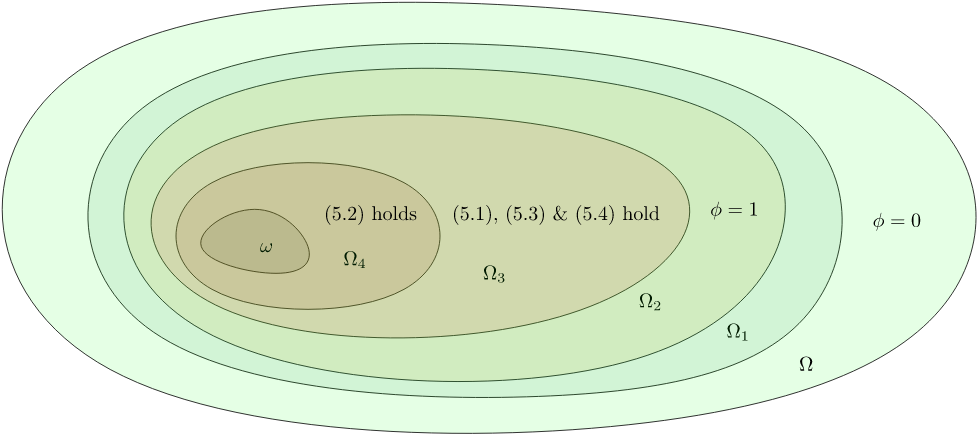}
        \caption{The nested set sequence used in the proof of Theorem \ref{thm:local_reg}.}
        \label{fig:nested_sequence}
    \end{figure}
    \begin{align}\label{eq:poincare_bva}
        \| v \|_{L^p(\tilde \omega)} \leq C(\tilde \omega)(\rho(v) + |\A v|(\widetilde \Omega)).
    \end{align}
    Now, let us consider nested domains 
    $$\omega \subset \subset \Omega_4 \subset \subset \Omega_3 \subset \subset \Omega_2 \subset \subset \Omega_1 \subset \subset \Omega$$ with $\Omega_i$ being Lipschitz for every $i = 1, \ldots, 4$, and without loss of generality assume $\supp \, \psi^k \subset \Omega_3$. 
    Given a function $\phi \in \mathcal{C}^\infty(\Omega)$ such that $\phi = 1$ on $\Omega_2$ and $\phi = 0$ on $\Omega \setminus \Omega_1$, we define $w := \phi u$. For every $\e > 0$ consider the family of mollifiers $\phi_\e \in \mathcal{C}^\infty(\R^n)$. Notice that $w_\e := \phi_\e \ast w \in \mathcal C^\infty(\Omega; V)$ is well-defined for sufficiently small $\e > 0$, and that 
    \begin{align}\label{eq:restriction_of_rho}
        \rho(w_\e) = \rho(w_\e|_{\Omega_3}).
    \end{align}
    Furthermore, we notice that
    \begin{align}
        (\A w_\e)(x) = \int_{B_\e(x)}\varphi_\e \, d \A u  
    \end{align}
    for $x \in \Omega_3$ and $\e > 0$ sufficiently small. By combining \eqref{eq:poincare_bva} with \eqref{eq:restriction_of_rho}, we infer
    \begin{align*}
        \int_{\Omega_3} |w_\e| \, dx \leq C(\Omega_3)(\rho(w_\e) + |\A w_\e |(\Omega_3)) \leq  C(\Omega_3) \left(\rho(w_\e) + | \A u |(\Omega) \right).
    \end{align*}
    and, therefore by \eqref{eq:poincare_bva_2},
    \begin{align*}
        \| w_\e \|_{L^p(\Omega_4)} \leq C(\Omega_3)(\rho(w_\e) + |\A w_\e|(\Omega_3)) \leq  C(\Omega_3) \left(\rho(w_\e) + | \A u |(\Omega) \right)
    \end{align*}
    Since $w_\e \to w$ in $\D'(\Omega, \R^n)$ we obtain
    \begin{align*}
        \rho(w_\e) &= \sum_{i = 0}^m \sum_{k \in \N} \frac{1}{2^kc_{i,k}} \left| \int_\Omega \psi^k b_iw_\e \, dx \right| \\
        & \hspace{2cm} = \sum_{i = 0}^m \sum_{k \in \N} \frac{1}{2^kc_{i,k}} \left| \langle w_\e, \psi^k b_i \rangle \right| \xrightarrow{\e \to 0} \sum_{i = 0}^m \sum_{k \in \N} \frac{1}{2^kc_{i,k}} \left| \langle w, \psi^k b_i \rangle \right|.
    \end{align*}
        
    Since by construction
    \[
        \sum_{i = 0}^m \sum_{k \in \N} \frac{1}{2^kc_{i,k}} \left| \langle w, \psi^k b_i \rangle \right| \leq 2m,
    \]
    we have, in particular,  
    \[
        \| w_\e \|_{BV^\A(\Omega_4; V)} < C
    \]
    with a constant $C$ only depending on the nested domains.
    Applying the compactness result, namely Theorem \ref{thm:compactness}, and due to $w_\e \to w$ in $\D'(\Omega_4; V)$ we infer $w \in BV^\A(\Omega_4; V)$. Since the nesting open sets can be chosen arbitrarily, we conclude $u \in BV^\A_{\rm loc}(\Omega; V)$.
\end{proof}

\begin{remark}
    We note here that the proof of Theorem \ref{thm:local_reg} only requires a compactness result which is similar to Theorem \ref{thm:compactness}. In fact, we do not use slicing techniques, but rely only on basic distribution theory. We want to emphasize that the proof uses a different approach to the usual reference for plasticity (cf.\ \cite{temam}), which relies more on classical elasticity theory and the specific form of the symmetrized gradient.
\end{remark}

Now, we have all the tools ready to finally prove Theorem \ref{thm:DR}.

\begin{proof}[Proof of Theorem \ref{thm:DR}]
    By Theorem \ref{thm:local_reg}, we have $u \in L^1_{\loc}(\Omega;V)$. We have to show integrability up to the boundary. Let $x_0 \in \partial\Omega \setminus \mathcal{P}$, where $\mathcal{P}$ denotes the set of non-differentiability points of $\partial\Omega$. Denote with $\nu(x_0)$ the normal to the boundary of $\Omega$ at the point $x_0$. Without loss of generality, we assume that $\nu(x_0) = e_n$. Note that the set $\mathcal{P}$ has Lebesgue measure zero by Sard's theorem. Now, let $Q$ be a cube centred at $x_0$ such that $\Sigma := Q \cap \partial \Omega$ can be written as a graph with respect to its basis, which is possible since $x_0 \notin \mathcal P$. Without loss of generality, using the notation $x' = (x_1, \ldots, x_{n-1})$, we write 
    \begin{equation*}
        \Sigma = \{(x',x_n) : x' \in \mathcal{O}, x_n = a(x')\},
    \end{equation*}
    where $\mathcal{O}$ is a cube in $\R^{n-1}$ and $a$ is a Lipschitz function which is strictly positive on $\overline{\mathcal{O}}$, and we suppose that
    \begin{equation}
    \label{inf_normal}
        \inf_{x \in \Sigma \setminus \mathcal{P}} \nu_n(x) =: \rho_1 > 0.
    \end{equation}
    For every $\alpha > 0$, we introduce the sets
    \begin{equation*}
        \Sigma_\alpha := \{(x',x_n) : x' \in \mathcal{O}, x_n = a(x') - \alpha\}
    \end{equation*}
    and
    \begin{equation*}
        F_\alpha := \{(x',x_n) : x' \in \mathcal{O}, a(x') - \alpha < x_n < a(x')\},
    \end{equation*}
    see Figure \ref{fig:sets}.

    \begin{figure}[H]
        \centering
        \pgfdeclarelayer{background layer}
        \pgfdeclarelayer{foreground layer}
        \pgfsetlayers{background layer,main,foreground layer}

        \begin{center}
        \begin{tikzpicture}[scale=1]
        
        \draw[thick,->] (-0.25,0)--(5,0);
        \draw[thick,->] (0,-0.25)--(0,4);

        \draw[smooth,domain=0:5,variable=\x] plot (\x, {sqrt(\x)+1});
        \draw[thick,smooth,domain=1:4,variable=\x] plot (\x, {sqrt(\x)+1});
        \filldraw (1,2) circle (1.5pt);
        \filldraw (4,3) circle (1.5pt);

        \draw[thick,smooth,domain=1:4,variable=\x] plot (\x, {sqrt(\x)});
        \filldraw (1,1) circle (1.5pt);
        \filldraw (4,2) circle (1.5pt);

        \draw[dashed] (1,2)--(1,0);
        \draw[dashed] (4,3)--(4,0);
        \filldraw (1,0) circle (1.5pt);
        \filldraw (4,0) circle (1.5pt);
        
        \begin{pgfonlayer}{foreground layer}
        \draw (-0.3,4) node[scale=0.85] {$x_n$};
        \draw (5,-0.2) node[scale=0.85] {$x'$};
        \draw (2.5,-0.2) node[scale=0.85] {$\mathcal{O}$};

        \draw (3.5,3.1) node[scale=0.85] {$\Sigma$};
        \draw (3.5,1.6) node[scale=0.85] {$\Sigma_\alpha$};
        \draw (2.5,2.1) node[scale=0.85] {$F_\alpha$};
        \end{pgfonlayer}
        \end{tikzpicture}
        \end{center}
        
        \caption{The sets $\Sigma_\alpha$ and $F_\alpha$ for $\alpha>0$ small.}
        \label{fig:sets}
    \end{figure}
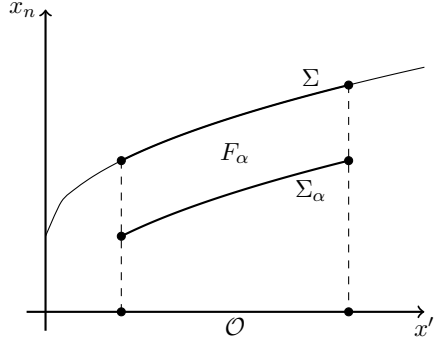
    
    Since $\A$ satisfies the rank-one condition and is elliptic, we find by Lemma \ref{lem:directional_spectrum_surjectivity_projections} $v^* \in V^*$ and $w^* \in W^*$ with $|w^*| = 1$ such that 
    \begin{align}
        \langle w^* , \AA(\xi) v \rangle  = \langle e_n , \xi \rangle \langle v^* , v \rangle
    \end{align}
    for every $v \in V$ and $\xi \in \s^{n-1}$. By Proposition \ref{prop:slicing_theory}, we have 
    \[
        (t \mapsto \langle v^*, u(x', t)\rangle)  \in BV_{\rm loc}(\Omega^{e_n}_{x'})
    \]
    for almost every $x' \in \mathcal O$. From this, we infer for every $\alpha \in (0, \alpha_0)$ and a.e. $x' \in \mathcal{O}$ 
    \begin{equation}
    \label{slicing}
        |\langle v^*, u(x',a(x')-\alpha_0) - u(x',a(x')-\alpha) \rangle| \le |\langle v^*, \partial_n u(x',\,\cdot\,) \rangle|(a(x')-\alpha_0, a(x')-\alpha),
    \end{equation}
    where the right-hand side is intended as the total variation of the measure $\langle v^*,  \partial_n u(x', \,\cdot\, ) \rangle$ over an interval. For $\alpha_0$ sufficiently small (namely such that $\Sigma_\alpha, \Sigma_{\alpha_0} \subset \Omega \cap Q$) we apply the co-area formula to infer
    \begin{equation*}
    \begin{aligned}
        &\int_{F_{\alpha_0} \setminus F_\alpha} | \langle v^*,  u(x) \rangle| \, dx \\
        &= \int_\alpha^{\alpha_0} \left( \int_{\Sigma_{\alpha'}} |\langle v^*,  u(x) \rangle| \, \frac{d\H^{n-1}(x)}{\nu_n(x)} \right) \, d\alpha' \\
        &= \int_\alpha^{\alpha_0} \left( \int_\Sigma | \langle v^*, u(x',a(x') - \alpha') \rangle| \, \frac{d\H^{n-1}(x)}{\nu_n(x',a(x'))} \right) \, d\alpha' \\
        &\le \frac{\alpha_0 - \alpha}{\rho_1} \int_{\Sigma_{\alpha_0}} |\langle v^*,  u(x) \rangle| \, d\H^{n-1}(x) + (\alpha_0 - \alpha) \int_\mathcal{O} |\langle v^*, \partial_n u(x',\cdot) \rangle|(a(x')-\alpha_0, a(x')-\alpha) \, dx',
    \end{aligned}
    \end{equation*}
    where in the last line we used \eqref{inf_normal} and \eqref{slicing}. Let now $\widetilde{\Omega} \subset\subset \Omega$ be open and Lipschitz such that $\widetilde{\Omega} \supset \mathcal{O} \times (\alpha_0,\alpha_1) = F_{\alpha_1} \setminus F_{\alpha_0}$ with $\alpha_1 > \alpha_0$ chosen such that $\Sigma_{\alpha_1} \subset \Omega$, see Figure \ref{fig:sets_2}.

    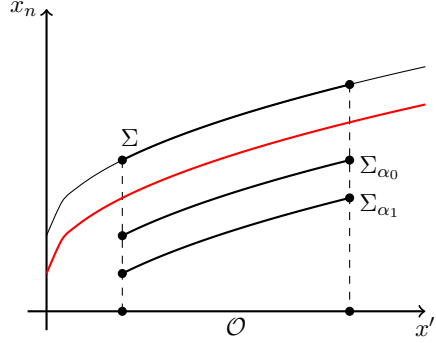
\begin{figure}[H]
        \centering
        \pgfdeclarelayer{background layer}
        \pgfdeclarelayer{foreground layer}
        \pgfsetlayers{background layer,main,foreground layer}

        \begin{center}
        \begin{tikzpicture}[scale=1]
        
        \draw[thick,->] (-0.25,0)--(5,0);
        \draw[thick,->] (0,-0.25)--(0,4);

        \draw[smooth,domain=0:5,variable=\x] plot (\x, {sqrt(\x)+1});
        \draw[thick,smooth,domain=1:4,variable=\x] plot (\x, {sqrt(\x)+1});
        \filldraw (1,2) circle (1.5pt);
        \filldraw (4,3) circle (1.5pt);

        \draw[thick,smooth,domain=1:4,variable=\x] plot (\x, {sqrt(\x)});
        \filldraw (1,1) circle (1.5pt);
        \filldraw (4,2) circle (1.5pt);

        \draw[thick,smooth,domain=1:4,variable=\x] plot (\x, {sqrt(\x)-0.5});
        \filldraw (1,0.5) circle (1.5pt);
        \filldraw (4,1.5) circle (1.5pt);

        \draw[dashed] (1,2)--(1,0);
        \draw[dashed] (4,3)--(4,0);
        \filldraw (1,0) circle (1.5pt);
        \filldraw (4,0) circle (1.5pt);

        \draw[thick,red,smooth,domain=0:5,variable=\x] plot (\x, {sqrt(\x)+0.5});
        
        \begin{pgfonlayer}{foreground layer}
        \draw (-0.3,4) node[scale=0.85] {$x_n$};
        \draw (5,-0.2) node[scale=0.85] {$x'$};
        \draw (2.5,-0.2) node[scale=0.85] {$\mathcal{O}$};

        \draw (1.1,2.3) node[scale=0.85] {$\Sigma$};
        \draw (4.4,1.9) node[scale=0.85] {$\Sigma_{\alpha_0}$};
        \draw (4.4,1.4) node[scale=0.85] {$\Sigma_{\alpha_1}$};
        \end{pgfonlayer}
        \end{tikzpicture}
        \end{center}
        
        \caption{In red the boundary of the set $\widetilde{\Omega}$.}
        \label{fig:sets_2}
    \end{figure}

    Since we know that $u \in L^1_\loc(\Omega; V)$, by the trace theorem for $BV^\A(\widetilde \Omega, V)$ (cf.\ \cite{BDG20}) we derive
    \begin{equation*}
        \int_{\Sigma_{\alpha_0}} |\langle v^*,  u(x) \rangle| \, d\H^{n-1}(x) \le |v^*| \|u\|_{BV^\A(\widetilde{\Omega};V)}.
    \end{equation*}
    Furthermore, estimating the second integral on the right-hand side by applying Proposition \ref{prop:slicing_theory} we infer
    \begin{equation*}
        \int_\mathcal{O} |\langle v^*, \partial_n u(x',\cdot) \rangle|(a(x')-\alpha_0, a(x')-\alpha) \, dx' \le |\A u|(F_{\alpha_0} \setminus F_\alpha) \le |\A u|(\Omega).
    \end{equation*}
    Combining these last two estimates, we obtain
    \begin{equation*}
        \int_{F_{\alpha_0} \setminus F_\alpha} |\langle v^*,  u(x) \rangle| \, dx \le |v^*| \, \frac{\alpha_0 - \alpha}{\rho_1} \|u\|_{BV^\A(\widetilde{\Omega}; V)} + (\alpha_0 - \alpha) |\A u|(\Omega),
    \end{equation*}
    and passing to the limit as $\alpha \to 0$ we derive that $\langle v^*,  u \rangle \in L^1(F_{\alpha_0})$.
    \\[0.5em]
    By Lemma \ref{lem:directional_spectrum_surjectivity_projections}, we have that $(\zeta_i, v_i^*) \in \partial \sigma (A)$ for some $v_i^*$ such that $v^*, v_1^*, .., v_{N-1}^*$ form a basis of $V^*$. By Lemma \ref{lem:polarization}, let $\gamma_i$ be such that 
    \[
        (e_n + \lambda \zeta_i, v^* + \gamma_i \lambda v_i^*) \in \partial \sigma(\A)
    \]
    for every $\lambda > 0$. Notice that for small $\lambda > 0$ there exists a cube $Q_i$ centered in $x_0$ with two side orthogonal to $\nu_i := e_n + \lambda \zeta_i/|e_n + \lambda \zeta_i|$ such that $Q_i \subset Q$ and $\partial \Omega \cap Q_i$ be written as a Lipschitz graph with respect to the basis of $Q_i$. Repeating the same analysis while substituting the pair $(e_n, v^*)$ with $(\nu_i, v_i^*)$ we derive that $\langle v_i^*, u \rangle \in L^1(Q_i \cap \Omega)$, meaning on the intersection $\tilde Q = Q \cap \bigcap_{i = 1}^{N - 1} Q_i$ we have $u \in L^1(\tilde Q)$. By a standard compactness argument, we cover $\partial \Omega$ with finitely many such neighborhoods $Q_i$ to show integrability near the boundary.
\end{proof}


\section*{Acknowledgements}

This research was funded in part by the Austrian Science Fund (FWF) projects \href{https://doi.org/10.55776/F65}{10.55776/F65} and \href{https://doi.org/10.55776/Y1292}{10.55776/Y1292}. For the purpose of open access, the authors have applied a CC BY public copyright license to any Author Accepted Manuscript (AAM) version arising from this submission.


\printbibliography

\end{document}